\documentclass[graybox]{svmult}

\usepackage{mathptmx}       
\usepackage{helvet}         
\usepackage{courier}        
\usepackage{type1cm}        
\usepackage{makeidx}         
\usepackage{graphicx}        
\usepackage{multicol}        
\usepackage[bottom]{footmisc}
\smartqed

\makeindex             

\usepackage{amsmath}
\usepackage{amssymb}
\usepackage[english]{babel}
\usepackage{setspace}
\allowdisplaybreaks[3]

\usepackage{paralist}
\usepackage{mathtools}
\usepackage{microtype}
\usepackage{enumerate}
\usepackage{mathrsfs}
\usepackage{xcolor}
\usepackage{paralist}
\usepackage[normalem]{ulem}
\usepackage{tikz}
\usetikzlibrary{matrix,arrows,positioning}
\usepackage[square,numbers]{natbib}
\usepackage{cleveref}
\newtheorem{prop}[theorem]{Proposition}
\newtheorem{defn}[theorem]{Definition}

\DeclarePairedDelimiter{\abs}{\lvert}{\rvert}
\DeclarePairedDelimiter{\sprod}{\langle}{\rangle}

\newcommand{\Sp}[2]{\sprod{#1 , #2}}
\newcommand{\laur}[1]{(\!(#1)\!)}

\newcommand{\K}{\mathbb{K}}
\newcommand{\C}{\mathbb{C}}
\newcommand{\Q}{\mathbb{Q}}
\newcommand{\Z}{\mathbb{Z}}

\newcommand{\M}{\mathcal{M}}

\newcommand{\la}{\mathfrak{a}}

\newcommand{\im}{\operatorname{im}}
\newcommand{\id}{\operatorname{id}}
\newcommand{\A}{\mathscr{A}}

\newcommand{\factor}[2]{\left.\raisebox{.2em}{$#1$}\middle/\raisebox{-.2em}{$#2$}\right.}
\newcommand{\ph}{[[ \lambda ]]}
\newcommand{\g}{\mathfrak{g}}
\newcommand{\h}{\mathfrak{h}}
\newcommand{\ad}{\operatorname{ad}}

\newcommand{\Hom}{\operatorname{Hom}}

\newcommand{\Homgr}{\operatorname{Homgr}}
\newcommand{\Der}{\operatorname{Der}}
\newcommand{\Ind}{\operatorname{Ind}}
\newcommand{\End}{\operatorname{End}}
\renewcommand{\id}{\operatorname{id}}
\newcommand{\U}{U}
\newcommand{\opp}{\mathrm{opp}}

\newcommand{\otimeshat}{\mathbin{\hat\otimes}}
\newcommand{\ev}{\operatorname{ev}}
\newcommand{\Vect}{\mathbf{Vec}}
\newcommand{\grVec}{\mathbf{grVec}}

\newcommand{\qbinom}[2]{\begin{bmatrix} #1 \\ #2 \end{bmatrix}_t}

\author{Benedikt Hurle  \and Abdenacer Makhlouf}
\title{Quantization of color Lie bialgebras}
\institute{
 B. Hurle  \at Universit\'{e} de Haute-Alsace, IRIMAS-d\'epartement de math\'ematiques, 6 bis rue des Fr\`{e}res Lumi\`{e}re, 68093 Mulhouse, France,\\ \email{benedikt.hurle@uha.fr}
\and A. Makhlouf \at Universit\'{e} de Haute-Alsace, IRIMAS-d\'epartement de math\'ematiques, 6 bis rue des Fr\`{e}res Lumi\`{e}re, 68093 Mulhouse, France,\\ \email{abdenacer.makhlouf@uha.fr}
}

\begin{document}


\maketitle

\abstract{
The main purpose  of this paper is to study Quantization of color Lie bialgebras, generalizing to color case the approach by Etingof-Kazhdan which were considered for superbialgebras by Geer.  Moreover we discuss Drinfeld category, Quantization of Triangular color Lie bialgebras and Simple color Lie bialgebras of Cartan type.} 

\tableofcontents
 \markboth{\leftmark}{\rightmark}
 

\section*{Introduction}

Lie bialgebras appeared in the 80th  mostly due  to V. G. Drinfeld \cite{Drinfeld-LieBialgebra,quasiLie} and M. A. Semenov-Tian-Shansky \cite{semenov}, who introduced the Poisson-Lie groups and discussed the relationships with the concept of a classical  $r$-matrix and Yang-Baxter equation. The Lie algebra  of a Poisson-Lie group  has a natural structure of Lie bialgebra, the Lie group structure gives the Lie bracket as usual, and the linearization of the Poisson structure on the Lie group  gives the Lie bracket on 
the dual of the Lie algebra.  Deformations and their relationships with cohomology of Lie algebras were discussed  by Nijenhuis-Richardson, following Gerstenhaber's approach, in \cite{NR} while cohomology of Lie bialgebras where studied  first in \cite{lecomte}.  One of the main problems formulated by Drinfeld in quantum group theory was  the existence of a universal quantization for Lie bialgebras, which arose from the problem of quantization of Poisson Lie groups. This issue was solved by Etingof-Kazhdan in  \citep{ek1} using the methods and ideas from \cite{KL}.

 Lie superbialgebras and Poisson-Lie supergroups were studied in \cite{andrus}. Quantized  functors have been constructed for  Lie superbialgebras and group Lie bialgebras in   \cite{Enriquez2008,geer}. 
Color Lie algebras are a natural generalization of superalgebra where the grading is defined by any abelian group and a commutation factor. They have become an interesting subject of mathematics and physics. The first study and cohomology  of color Lie algebras were  considered by Scheunert  in \cite{scheunert1,scheunert98}.  Moreover, various properties were studied in color setting, see \cite{colorlie} and references therein.
 
 We aim in this paper to discuss quantization of color Lie bialgebras, by adapting the method of Etingof-Kazhdan \citep{ek1}, which were already considered in the super-case by Geer  in \citep{geer}. In the first section, we provide some preliminaries about colored structures and quantized universal enveloping algebras. In in Section 2 Drinfeld associator and category are discussed. Section 3 includes the main results about Quantization of color Lie bialgebras. Section 4 deals with triangular color Lie bialgebras and in Section 5 a second quantization is presented. The last section provides a discussion about simple color Lie bialgebras of Cartan Type.

\section{Preliminaries}

In this section we give the basic definitions of color vector space, Lie bialgebra and so on. We also make some remarks regarding category theory, which we will need later to use in the construction of Etingof-Kazhdan.
In fact, we will need a bit of enriched category theory see e.g. \citep{kelly}.

Let $\K$ be a field of characteristic 0.

\subsection{Graded and Color vector spaces}

\begin{defn}[Commutation factor]
Let $\Gamma$ be an abelian  group. Then a map $\epsilon: \Gamma \times \Gamma \rightarrow \K^\times$ is called an anti-symmetric bicharacter or commutation factor if
\begin{align}
\epsilon(f+g,h) &= \epsilon(f,h)  \epsilon(g,h), \\
\epsilon(f,g+h) &= \epsilon(f,g) \epsilon(f,h), \\
\epsilon(f,g)\epsilon(g,f) &= 1. 
\end{align}
\end{defn}
We note that the product of two commutation factors is again a commutation factor.

Let $\Gamma$ be an abelian group, and for each $g \in \Gamma$ let $V_g$ be a vector space, then we call $V = \bigoplus_{g \in G} V_g$ a $\Gamma$-graded vector space.

\begin{defn}
A $(\Gamma,\epsilon)$ color vector space is a $\Gamma$-graded vector space $V$, together with a commutation factor $\epsilon$.
\end{defn}

In the following a color vector space will always be with respect to a fixed $\Gamma$ and $\epsilon$.

Given an abelian group $\Gamma$ and a commutation factor $\epsilon$ on  it, one can define the category $\grVec$ of $(\Gamma,\epsilon)$-color vector spaces. This category is closed monoidal, where the tensor product is given by 
\begin{equation}
(V \otimes W)_i = \bigoplus_{j+k=i} V_j \otimes W_k,
\end{equation}
and the internal homs, called the graded morphisms, are given by
\begin{align}
& \Homgr(V,W) = \bigoplus_{i \in \Gamma} \Homgr^i (V,W), \\  &  \Homgr^i (V,W) = \{ \phi: V \rightarrow W \text{ linear} | \phi(V_j) \subset W_{j+i} \}.
\end{align}
It is clear that $\Homgr(V,W)$, is again a color vector space.
The morphisms in the ordinary category $\grVec$ are given by the graded morphisms of degree 0.  We further define the flip $\tau_{V\otimes W}: V\otimes W \to W \otimes V$ by $v \otimes  w \mapsto \epsilon(\abs{v},\abs{w}) w\otimes v$, if the involved spaces are clear we simply write $\tau$. With this $\grVec$ becomes a symmetric monoidal category.
We also note that $\grVec$ is an enriched category over $\Vect$.
Another, maybe nicer, way to define it is to consider the category $\hat\Gamma$,
enriched over $\Vect$, which has as objects the elements of $\Gamma$, and as morphisms $\Hom(g,h) = \K$ for all $g,h \in \Gamma$. One can make $\Gamma$ into a monoidal category by defining the tensor product to be the addition. Then one can consider the associator to be trivial. A commutation factor gives a symmetry on this category. With this structure the $\Vect$-functor category $\Hom(\hat \Gamma,\Vect)$ is again a symmetric monoidal category and isomorphic to $\grVec$.

Throughout this paper, we will use  the Koszul rule, this means given graded morphisms 
$\phi$ and $\psi$ we have $( \phi \otimes \psi) (x \otimes y) = \epsilon(\psi,x) \phi(x) \otimes \psi(y)$, with  $\epsilon(\phi , x)$ denoting $\epsilon(\abs{\phi},\abs{x})$ for short and  where $\abs{\cdot}$ is  the degree.

\subsection{Color Lie algebras}

In this section we  give the basic definitions of color Lie algebras and bialgebras, for more details see e.g. \citep{colorlie}.

\begin{defn}[Color Lie algebra]
For a group $\Gamma$ and a commutation factor $\epsilon$,  a $(\Gamma,\epsilon)$-color Lie algebra is a $\Gamma$-graded vector space $\g$ with a graded bilinear map $[\cdot,\cdot]: \g \times \g \rightarrow \g$ of degree zero, such that  for any homogeneous elements $a,b,c \in \g$ 
\begin{gather}
 [a,b] = - \epsilon(a,b)[b,a], \\
j(a,b,c) := \epsilon(c,a) [a,[b,c]] +\epsilon(a,b) [b,[c,a]] +\epsilon(b,c) [c,[a,b]]  =0.
\end{gather}
\end{defn}
The second equation is called color Jacobi identity. It  can also be written as 
\begin{equation}
[a,[b,c]] = [[a,b],c] + \epsilon(a,b)[b,[a,c]], 
\end{equation} 
which shows that the adjoint representation $\ad_a (b) := [a,b]$ for $a \in \g$ is a color Lie algebra derivation.
Using $\sigma (a \otimes b \otimes c) = \epsilon(a,bc) b \otimes c \otimes a$, $\beta(a,b) = [a,b]$ and the Koszul rule, it can also be written as
\begin{equation}
\beta(\beta \otimes \id) (\id + \sigma + \sigma^2) =0.
\end{equation}

A morphism $\phi$ of color Lie algebras $(\g,[\cdot,\cdot])$ and $(\h, [\cdot,\cdot]')$ is a morphism of color vector spaces such that 
\begin{equation}
[\phi(x), \phi(y)]' = \phi([x,y]).
\end{equation}
If $\phi$ is homogenous, from this one gets that $\deg \phi + \deg \phi = \deg \phi$ and so $\deg \phi= 0$.  

An ideal of a color Lie algebra $\g$ is a graded subspace $\mathfrak{i}$ such that $[\mathfrak{i},\g] \subset \mathfrak{i}$.
We call a color Lie algebra simple if it has no color Lie ideal. Note that it can have a non graded ideal. 
If $\mathfrak{i} \subset \g$ is a color Lie ideal then  the quotient $\factor{\g}{\mathfrak{i}}$ is again a color Lie algebra.
A color Lie subalgebra is a graded subspace $\h$ such that $[\h,\h] \subset \h$.

Let $A$ be a $\Gamma$-graded associative algebra and $\epsilon$ a commutation factor then  
\begin{equation}
[a,b] = ab - \epsilon(a,b) ba
\end{equation}
defines a color Lie algebra structure. We denote the corresponding color Lie algebra by $A_L$.
So we get especially a Lie bracket on the graded homomorphisms of a color vector space, for which we have:
\begin{prop}
The color derivations $\Der(A)$ of a color algebra $A$ form a color Lie subalgebra of $\Homgr(A)$.
\end{prop}

\begin{defn}[Universal enveloping algebra] \label{de:cuea}
For a color Lie algebra $\g$ one defines the universal enveloping algebra, or short UEA, $U(\g)$ by the tensor algebra $T(\g)$  modulo the ideal generated by elements of the form
\begin{equation}
 x y- \epsilon(x,y)yx  - [x,y] 
\end{equation}
for $x,y \in \g \subset U(\g)$.
\end{defn}

\begin{theorem}[\citep{scheunert1}]
The universal enveloping algebra   $U(\g)$  is an  (filtered) associative color  algebra. With the graded commutator it is a color Lie-algebra, with $\g$ as a color Lie subalgebra. It has the universal property that is for any color algebra $A$ and color Lie algebra homomorphism $f: g \rightarrow A_L$ there exits a unique algebra homomorphisms such that $f = g|_\g$.
\end{theorem}

In fact one has the structure of a color Hopf algebra on $U(\g)$. 
The coproduct is given by $\Delta(x) = x \otimes 1 + 1 \otimes x$ for $x \in \g \subset U(\g)$ and extended to the rest of  $U(\g)$ by the universal property stated in the theorem above.

The Lie algebra $\g$ is precisely formed by the primitive elements in $U(\g)$, where an element $x \in U(\g)$ is called primitive if it satisfies $\Delta(x) = 1\otimes x +x \otimes 1$.

\subsection{Color Lie bialgebras}

For the definition of color Lie bialgebras, we first need to define the adjoint action on tensor powers of $\g$. Let $a\in \g$ and $x \otimes y \in \g \otimes \g$ then, we set 
\begin{equation}
  a \cdot ( x\otimes y) = (a \cdot x) \otimes y + \epsilon(a,x) x \otimes (a \cdot y).
\end{equation}
This can be generalized to higher tensor products.

\begin{defn}[Color Lie bialgebra]
A color Lie bialgebra is a color Lie algebra $\g$ with a color antisymmetric  cobracket $\delta: \g \rightarrow \g \otimes \g$ of degree 0, such that the compatibility condition
\begin{align}
\delta( [a,b]) &= a \cdot \delta(b) - \epsilon(a,b) b \cdot \delta(a)
\end{align}
holds and $\delta$ satisfies the co-Jacobi identity given by
 \begin{equation} \label{eq:cojac}
(\id + \sigma + \sigma^2)(\delta \otimes \id)\delta =0.
 \end{equation}
\end{defn}

It is called quasitriangular, if there exits an $r \in \g \otimes \g$, such that $r + \tau(r)$ is $\g$ invariant and satisfies
$\operatorname{CYB}(r)=0$
 with 
\begin{equation}\label{eq:yb}
\operatorname{CYB} = [r_{12},r_{13}]+[r_{12},r_{23}]+[r_{13},r_{23}].
\end{equation}
It is called triangular if in addition, $r = -\tau(r)$.

\begin{defn}[Color Manin triple]
A  color Manin triple is a triple $(\mathfrak{p},\mathfrak{p}_+,\mathfrak{p}_-)$, where $\mathfrak{p}$ is a color Lie algebra, $\mathfrak{p}_\pm$  are color Lie subalgebras of $\mathfrak{p}$ and $\mathfrak{p} = \mathfrak{p}_+ \oplus \mathfrak{p}_-$ as color vector spaces, with a non-degenerate invariant symmetric inner product $(\cdot,\cdot)$ on $\mathfrak{p}$, such that $\mathfrak{p}_\pm$ are isotropic, i.e. $(\mathfrak{p}_\pm,\mathfrak{p}_\pm)=0$.
Invariant here means that
\begin{equation}
([a,b],c) + \epsilon(a,b) (b,[a,c]) =0,
\end{equation}
and  symmetric means that  
\begin{equation}
(a,b) = \epsilon(a,b)(b,a).
\end{equation}
\end{defn}
Note that the invariance can also be written as
\begin{equation}
([b,a],c) = ( b,[a,c]).
\end{equation}

\begin{theorem}
Let $\g$ be a color Lie bialgebra and  set $\mathfrak{p_+}= \g, \mathfrak{p_-}= \g^*$ and $\mathfrak{p}= \mathfrak{p_+} \oplus \mathfrak{p_-}$. Then $(\mathfrak{p},\mathfrak{p_+},\mathfrak{p_-})$ is  a color Manin triple. Conversely any finite-dimensional color Manin triple $\mathfrak{p}$ gives rise to Lie bialgebra structure on $\mathfrak{p_+}$.  
\end{theorem}

In  the following  let $\{x_i\}_i$ be a basis of  $\g$ and $\{\alpha^i\}_i$ be the corresponding dual basis of $\g^*$, i.e. $\alpha^i(x_j) = \delta^i_j$, where $\delta$ is the Kronecker delta.

\begin{prop}[Double]
If $\mathfrak{p}$ is finite dimensional, there is also the structure of a color Lie bialgebra on $\mathfrak p$ given by the $r$-matrix $r =x_i \otimes \alpha^i $.
For this we have $\delta(x) = x \cdot r = [x,x_i] \otimes \alpha^i + \epsilon(i,x) x_i \otimes [x,\alpha^i]$. This  color Lie bialgebra is called the double of $\g$, and it is denoted  by $D(\g)$.
\end{prop}
For explicit proofs, see e.g. \citep{colorlie}.

On $D(\g)$ we define the Casimir element $\Omega = r + \tau \circ r$, which is  invariant, i.e. $x \cdot \Omega = 0$ for all $x \in \g$.

We consider the category of all $(\Gamma,\epsilon)$ color Lie bialgebras $\mathbf{LBA}$. 
The morphisms are given by the graded morphisms, which preserve the color Lie bracket and cobracket. Note that this category is not enriched over $\grVec$, since the sum of two morphisms is in general not a morphism again.
It is also not necessary to consider graded morphisms, since for a homogeneous $\phi$ from $\phi([x,y])=[\phi(x),\phi(y)]$ we get $\deg(\phi) + \deg(\phi) = \deg(\phi)$ so we must have $\deg(\phi)=0$. So this can be considered as an ordinary category with morphism only of degree $0$.

\subsection{Topologically free modules}

For the quantization we  need $\K\ph$-modules, where $\K\ph$ denotes the ring of formal power series over $\K$.
A $\K\ph$-module is called topologically free if it is isomorphic to one of the form $V \otimeshat \K\ph$. Given a $\K$-module $V$, we will simply denote it by $V\ph$.
Here  $\otimeshat$ denotes the completed tensor product with respect to the filtration  by $\lambda$ or the $\lambda$-adic topology.

A graded $\K\ph$-module $V$ is called free, if all $V_i$ are free. 
In general this is not equivalent to the statement that $\bigoplus_{i \in \Gamma} V_i$ is free. But it is equivalent if only finitely many $V_i$ are nonzero.

\subsection{Quasitriangular color quasi-Hopf algebras}

\begin{defn}[Color quasi-Hopf algebra]
A color  quasi-Hopf algebra is an  associative color algebra $H$ with a multiplication $\mu$, a coproduct $\Delta$, a unit $1$, a counit $\varepsilon$ and  an  invertible  associator $\Phi \in H^{\otimes 3}$  all of degree $0$, which satisfy 
\begin{align*}
\forall x,y \in  H : \Delta(xy) &=\Delta(x) \Delta(y) && \text{(compatibility)}, \\
(\varepsilon \otimes \id)\Delta &= \id = (\id \otimes \varepsilon) \Delta && \text{(counit)}, \\
\Phi (\Delta \otimes \id ) \Delta &= (\id \otimes \Delta)\Delta \Phi && \text{(quasi- coassociativity)},\\
\Phi_{1,2,34} \Phi_{12,3,4} &=\Phi_{2,3,4}\Phi_{1,23,4}\Phi_{1,2,3} && \text{(Pentagon identity)},\\
(\id \otimes \epsilon \otimes  \id) \Phi &= 1 \otimes 1.
\end{align*}
Here we used the shorthand notation $\Phi_{1,23,4} = (\id \otimes \Delta \otimes \id )\Phi$ , $\Phi_{2,3,4} = 1 \otimes \Phi$ and similar for the others. 
It also has an antipode $S$, which satisfies $$\mu (\id \otimes S) \Delta = \id =\mu (\id \otimes S) \Delta.$$
\end{defn}
One could also allow for left and right unitors, but we will not do so. 

Note that since these equations use the product in $H \otimes H$, defined by $( a\otimes b)( c\otimes d) = \epsilon(b,c) ac \otimes bd$ for $a,b,c,d \in H$, they depend on the commutation factor $\epsilon$.

We assume also the operations to be of degree zero for two reasons, it is easier to handle this way categorically, and actually all operations have to be of degree zero, because they respect the unit. 

A color  quasi-Hopf algebra, where $\Phi = 1 \otimes 1 \otimes 1$ is simply called a color Hopf algebra.
A color quasi-Hopf algebra $H$ is called quasitriangular, if there exists also an $r$-matrix $R \in H^{\otimes 2}$, of degree $0$, such that 
\begin{align}
(\id \otimes \Delta)R &= \Phi^{-1}_{231} R_{13}\Phi_{213} R_{12} \Phi_{123}^{-1}, \\
(\Delta \otimes \id)R &= \Phi_{312} R_{13}\Phi_{132}^{-1} R_{23} \Phi_{123},  \label{eq:quasitri} \\
R \Delta^\opp &=  \Delta R.
\end{align}
Here $\Delta^\opp(x) = \tau \Delta(x)$,  $\Phi_{312} = \tau_{H,H \otimes H} \Phi$ and similar for other permutations.
It is called triangular if $R_{21} R  = \id$.

From \cref{eq:quasitri} and the fact the $R$ is invertible, it follows that $(\epsilon \otimes \id)R = 1 = (\id \otimes \epsilon) R$, so $R$ is automatically of degree 0.

Two quasitriangular quasi-Hopf algebras $H$ and $H'$ are called twist equivalent if there exists an invertible element  $J \in H^{\otimes 2}$ of degree $0$ and an algebra isomorphism $\theta: H \to H'$, such that 
\begin{align*}
(\epsilon \otimes \id)J &= 1 = ( \id \otimes \epsilon) J, \\
\Delta' &= J^{-1} ((\theta \otimes \theta) \Delta \theta^{-1}(x))J, \\
\Phi' &=  J_{2,3}^{-1} J_{1,23}^{-1} \theta(\Phi) J_{12,3} J_{1,2}, \\
R' &= J_{21}^{-1} R J.
\end{align*}

In the following we will be mostly interested in the case, where $\theta$ is the identity. 

If $J$ satisfies the first identity above one can define a new  twist equivalent quasitriangular quasi-Hopf algebra by using the identities as definitions for $\Delta',\mu'$ and $R'$.

\begin{theorem}
The category of  modules  over a quasitriangular color quasi-Hopf algebra is a braided monoidal category, enriched over $\grVec$. 

And if two Hopf algebras are twist equivalent, the category of modules are tensor equivalent, i.e. it exists an invertible monoidal functor between them.
\end{theorem}

There is a category of ((quasi)-triangular) color  quasi-Hopf algebras, where the morphisms preserve the product, coproduct, unit, counit and the associator. In the case of (quasi-)triangular-Hopf algebras they also preserve the $r$-matrix. Since the morphism satisfies $\phi(1)=1$ they have to be of degree 0 if they are homogeneous.  So this category is only an ordinary one.

\subsection{Quantized universal enveloping algebras}

\begin{defn}[QUEA]
Let $H$ be a topological free  color Hopf $\K\ph$-algebra. Then it is called a quantized universal enveloping algebra, if $\factor{H}{\lambda H} \cong U(\g)$, for a color Lie algebra $\g$.  If $\g$ is a color Lie bialgebra then  $H$ is a quantization of it  if in addition 
\begin{equation}
\delta = \frac{1}{\lambda} (\Delta -\Delta^\opp) \mod \lambda.
\end{equation}
\end{defn}

 Let $H$ be a quantization of a quasitriangular color Lie bialgebra. Then $(H,R)$ is called a quasitriangular quantization if $(H,R)$ is a quasitriangular Hopf algebra and  $ R \equiv 1 \otimes 1 + \lambda r \mod \lambda^2 $.


\section{Drinfeld category}

\subsection{Associator}

Following \citep{ek1}, let $T_n$ be the algebra over $\K$ generated by symmetric elements $t_{ij}$ for $1\leq i,j \leq n$, satisfying the relations $t_{ij}= t_{ji}$, $[t_{ij},t_{lm}]=0$ if $i,j,l,m$ are distinct and $[t_{ij},t_{ik}+ t_{jk}] =0$.
For disjoint sets $P_1, \dots, P_n \subset \{1,\ldots,m\}$, there exists a unique homomorphism $T_n \to T_m$ defined on generators by $ t_{ij} \mapsto \sum_{p \in P_i, q \in P_j} t_{pq}$. We denote it by $X \mapsto X_{P_1,\dots, P_n}$.
For $\Phi \in T_3$, the relation 
\begin{equation}
  \Phi_{1,2,34}\Phi_{12,3,4}=\Phi_{2,3,4}\Phi_{1,23,4}\Phi_{1,2,3}
\end{equation}is called the pentagon relation, and for $ R = e^{\frac{\lambda}{2}t_{12}} \in T_2\ph$ the relations
\begin{align}
  R_{12,3} = \Phi_{3,1,2}R_{1,3}\Phi^{-1}_{1,2,2}R_{2,3}\Phi_{1,2,3}
  R_{1,23} = \Phi^{-1}_{2,3,1}R_{12,3}\Phi_{2,1,3}R_{12,3}\Phi^{-1}_{1,2,3}
\end{align}
are called the hexagon relations.

There is the following well known theorem due to Drinfeld \citep{drinfeld1}:
\begin{theorem}
There exists an associator over $\Q$.
\end{theorem}

This means that there is also an associator for every field, wich includes the rational numbers. In the following we will fix such an associator for $\K$.


For a color Lie algebra $\g$, with a symmetric invariant element $\Omega= \Omega_1 \otimes \Omega_2$, we can define a map from $T_n$ to $\End(M_1 \otimes \dots \otimes M_n)$, by setting $t_{ij} \mapsto \Omega_{ij}$. Here $\Omega_{ij}$ is $1 \otimes \dots \otimes \Omega_1 \otimes \dots \otimes \Omega_2 \otimes \dots \otimes 1$ with the components of $\Omega$ in the $i$-th and $j$-th factor in the tensor product. If $i>j$, we have that $\Omega_{ij} = \tau \Omega_{ji}$. So the $\Omega_{ij}$ satisfy the same relations as the $t_{ij}$ since $\Omega$ is invariant. As in the non-color case we get the following theorem.
 
\begin{theorem}
We get a quasitriangular  color quasi-Hopf algebra $(U(\g)\ph ,\Delta,\varepsilon,\Phi,R)$, which we denote by $A_{\g,t}$.
\end{theorem} 

\subsection{Drinfeld category}

Let $\g_+$ be a finite dimensional color Lie algebra, and $\g = D(\g_+)= \g_+ \oplus \g_-$ be the Drinfeld double of $\g_+$ , with its Casimir $\Omega$. Since $\Omega$ is invariant and symmetric, we get a quasitriangular color quasi-Hopf algebra $A_{\g,\Omega}$. 

We define the category $\M_\g$, whose objects are $\g$-modules and whose morphisms are given by 
\begin{equation}
\Hom_{\M_\g} (V,W)= \Hom_\g(V,W)\ph.
\end{equation}
Note: We consider this category to be enriched over $\grVec$ so the homomorphisms here are graded in general.

We equip $\M_\g$ with the usual tensor product, symmetry given by $\beta_{V,W}: V \otimes W \to W \otimes V, v \otimes w \mapsto \tau \exp(\frac{\lambda\Omega}{2}) v \otimes w$
and associator
\begin{equation}
\Phi_{V,W,U}: (V\otimes W) \otimes U \to  V\otimes (W \otimes U), 
v \otimes w \otimes u \mapsto \Phi \cdot (v \otimes w \otimes u).
\end{equation}

This is called Drinfeld category. In fact it is just the category of   modules over the quasitriangular color quasi-Hopf algebra $A_{\g,\Omega}$.

\section{Quantization of color Lie algebras} \label{ch:quant1}

Let $\A$ be the category of topological free graded  $\K\ph$-modules, considered as $\grVec$-category.

Let $\g_+$ be a finite dimensional color Lie algebra and $\g =D(\g_+)$ be its double.
We consider the Verma modules 
\begin{align*}
M_+ &= U(g) \otimes_{U(\g_+)} c_+ & &\text{and}  & M_- &= U(g) \otimes_{U(\g_-)} c_-, 
\end{align*}
where $c_\pm$ is the trivial one dimensional $U(\g_\pm)$-module, concentrated in degree 0. 
The module structure on $M_\pm$ comes from its definition as Verma module by acting on $U(\g)$.

\begin{remark}
 It should be clear that given a graded algebra $A$, right $A$-module $M$ and a left $A$-module $N$, the tensor product $M \otimes_A N$ is well defined and again a graded vector space. If $M$, was 
in fact a $B$-$A$-bimodule, it is a left $B$-module, and similarly  for $N$. 

There is also a nice definition for the tensor product over $A$ using a coend. A graded algebra $A$ can be considered as a category $\A$ enriched  over $\grVec$ with only one object, which we denote by $*$. A left (resp. right)  $A$-module is precisely a functor  from $\A$ (resp. $\A^\opp$)  to $\grVec$. Then the tensor product over $A$ can be defined as 
\begin{equation}
M \otimes_A N := \int^{a \in \A} M(a) \otimes N(a),
\end{equation} 
where the integral symbol denotes a coend.
\end{remark}

Note: By the PBW Theorem, we have an isomorphism $U(\g_+) \otimes U(\g_-) \cong U(g)$ of vector spaces, which is given by the multiplication in $U(\g)$  and in general $U(\g \oplus \h) \cong U(\g) \otimes U(\h)$ as vector space.

This implies that 
\begin{equation}
M_\pm =  U(\g_\mp) 1_\pm  ,
\end{equation}
with $1_\pm$ in $M_\pm$. So $M_\pm$ are free $U(\g_\mp)$-modules.

\begin{lemma}\label{th:umiso}
There exists an isomorphism $\phi:U(\g) \to M_+ \otimes M_-$ of $U(\g)$-modules of degree 0 given on generators by $1\mapsto 1_+ \otimes 1_-$
\end{lemma}
\begin{proof}
It is well defined by the universal property of $U(\g)$, as the extension of $x \mapsto x 1_+ \otimes 1_- +  1_+ \otimes x1_-$.
It is an isomorphism since $U(g)$ and $M_\pm$ can be regarded as free connected coalgebras and $\phi$ clearly is an isomorphism on the primitive elements.
\end{proof}

Next we define a $\grVec$-functor $F: \M_\g \to \A$ by
\begin{equation}
F(V)= \Hom_{\M_\g}(M_+ \otimes M_-,V).
\end{equation}
Since this is just a Hom functor, its definition on morphisms is clear.

The isomorphism from \Cref{th:umiso} gives an isomorphism
\begin{equation}\label{eq:psi}
\Psi_V : F(V) \to V\ph,  f \mapsto f(1_+ \otimes 1_-) .
\end{equation}
So the functor $F$ is naturally isomorphic to the ``forgetful'' functor.

We want to show that $F$ is a tensor functor, for this we need a natural transformation 
$J: \otimes_{\K\ph} \circ (F \otimes F) \to   F \circ \otimes_{\M_\g}  $, which also satisfies $J_{U \otimes V ,W} \circ (J_{U,V} \otimes \id_W) = J_{U ,V  \otimes W} \circ (\id_U \otimes J_{V,W} )$.

Define $i_\pm : M_\pm \to M_\pm \otimes M_\pm$ by $1_\pm \mapsto 1_\pm \otimes 1_\pm$ and extended as $\g$-module morphism.
Clearly $i_\pm$ is of degree $0$.

\begin{lemma}
The map $i_\pm$ is coassociative, i.e. $\Phi \circ (i_\pm \otimes \id) \circ i_\pm = (\id \otimes i_\pm) \circ i_\pm $.
\end{lemma}
\begin{proof}
Following \citep[Lemma 2.3]{ek1}. 
We only prove the identity for $i_+$, since the proof for $i_-$ is analog.

Let $x \in M_+$. Then since the comultiplication in $U(\g_-)$ is coassociative, we have 
\begin{equation}
(i_+ \otimes \id) i_+ x = (\id \otimes i_+) i_+ x.
\end{equation}
It is enough to show
\begin{equation}
\Phi \cdot (i_+ \otimes \id) i_+ x = ( i_+ \otimes \id) i_+ x,
\end{equation}
but since $\Phi$ is $\g$ invariant by definition, it is enough to show this for $x= 1_+$. 
This means $\Phi \cdot 1_+ \otimes 1_+ \otimes 1_+ = 1_+ \otimes 1_+ \otimes 1_+ $. Which follows directly from the fact that $\Omega$ annihilates $1_+ \otimes 1_+$ and the definiton of $\Phi$.
\end{proof}

We define $J$ by 
\begin{equation}
J_{V,W}(v,w) = (v \otimes w) \circ \Phi^{-1}_{1,2,34} \circ ( \id \otimes \Phi_{2,3,4}) \circ \beta_{2,3} \circ ( \id \otimes \Phi^{-1}_{2,3,4}) \circ \Phi_{1,2,34} \otimes (i_+ \otimes i_-)
\end{equation}
for $v \in F(V), w \in F(W)$.
Since all involved maps are of degree 0, $J$ is also of degree 0.

The maps  can be given by a diagram
\begin{equation}
\begin{split}
M_+ \otimes M_- \rightarrow  (M_+ \otimes M_+) \otimes (M_- \otimes M_-))   
\rightarrow M_+ \otimes ((M_+ \otimes M_-) \otimes M_- )  \\
\rightarrow M_+ ((\otimes M_- \otimes M_+) \otimes M_- )   \rightarrow (M_+ \otimes M_-) \otimes (M_+ \otimes M_- ) \rightarrow V \otimes W
\end{split}
\end{equation}

Actually in the graded case there is a better definition 
\begin{align*}
J_{V,W} : \Hom(M_+ \otimes M_- ,V) \otimes \Hom(M_+ \otimes M_-,W) \xrightarrow{\otimes}
\Hom(M_+ \otimes M_- \otimes M_+ \otimes M_-, V \otimes W) \\ \xrightarrow{\Hom(\beta_{23},\cdot)}
\Hom(M_+ \otimes M_+ \otimes M_- \otimes M_-, V \otimes W) \xrightarrow{\Hom(i_+ \otimes i_-, \cdot)}
\Hom(M_+ \otimes  M_-, V \otimes W)
\end{align*}

\begin{theorem}
The functor $F$ together with the natural transformation $J$ forms a tensor functor.
\end{theorem}
\begin{proof}
  One only needs the check the equation 
  \begin{equation}
  F(\Phi_{UVW}) \circ J_{U \otimes V,W}  (J_{U,V} \otimes \id ) = J_{U,V \otimes  W} \circ ( \id \otimes J_{V,W}) .
  \end{equation}
  The proof given for this in \citep{etingofbook} is diagrammatically so it also holds in the color case. 
\end{proof}

Let $\End(F)$ be the color algebra  of natural endomorphisms of $F$.  In  fact what we mean here is 
\begin{equation}
\End(F) = \int_{ V \in \M_\g} \Hom( F(V) , F(V)).
\end{equation}
This is the end in the category $\grVec$ enriched over itself, so it is a color vector space. We need to use this definition since in the ``classical'' one the natural transformations must consist of morphisms of degree 0. A natural transformation in this sense consists of a family $\eta_V$ of graded maps $V \to V$, which satisfy the ``normal'' relation for a natural transformation but with additional sings, this means
\begin{equation}
\eta_W F(f) = \epsilon(\eta,f) F(f) \eta_V.
\end{equation}
 We say that 
 $\eta \in \End(V)$ is of degree $i$ if all $\eta_V$ are of degree $i$. So $\End(F) = \bigoplus_{g \in \Gamma} \End_g(F)$ is a graded vector space, where $\End_g(F)$ consists of the natural transformations of degree $g$.

\begin{prop}
There is a canonical color algebra isomorphism 
\begin{equation}
\Theta :  U(\g)\ph \to \End(F) , x \mapsto x \cdot
\end{equation}
where $x \cdot$ on the right stands for the action of $x$ induced on every  $U(\g)$-module.
\end{prop}
\begin{proof}
$\End(F)$ and $F(M_+ \otimes M_-)$ are isomorphic due to the Yoneda Lemma for enriched categories. Since $M_+ \otimes M_-$ is isomorphic to $U(\g)$ as $U(\g)$-module, $\End(F)$ is in fact isomorphic to $U(\g)$.

Following \citep{geer}: 
We can identify $F(V)$ with $V\ph$.  
The map $\Theta$ is injective since the action of $U(\g)$ on itself is injective. 
We want to show that is surjective. Let $\eta\in \End(F)$, and we identify $\eta_V$ with a map $V\ph \to V\ph$. We define $x := \eta_{U(\g)}(1)$. We claim that  $\eta_V = x \cdot$.  For $y \in  U(\g)$ we define $r _y \in \End(U(\g))$ by $r_y(z)= \epsilon(y,z) zy$ for $z \in U(\g)$. We have 
\begin{equation}
\eta_{U(\g)}(y) = \eta_{U(\g)} (r_y 1) =  \epsilon(x,y) r_y \eta_{U(\g)}(1) = \epsilon(x,y) r_y x = xy.
\end{equation}
So we get $\eta_{U(\g)}= x\cdot$.  Similarly one shows that  $\eta_{V}= x\cdot$ for any free $\g$-modules $V$ and since any $\g$-module is a quotient of a free one the claim follows.
\end{proof}

We define $J \in (U(\g) \otimes U(\g)) \ph$ by 
\begin{equation} \label{eq:defJ}
J = ( \phi^{-1} \otimes \phi^{-1})( \Phi^{-1} _{1,2,34} ( 1 \otimes \Phi_{2,3,4}) \beta_{2,3} ( \id \otimes \Phi_{2,3,4}) \Phi_{1,2,34} ( 1_+ \otimes 1_+ \otimes 1_- \otimes 1_- )).
\end{equation}
This means $J_{U(\g),U(\g)} (\phi^{-1} \otimes \phi^{-1})( 1_+ \otimes 1_-)$. With this the natural transformation $J$ can be identified with the action of the element $J$.

\begin{prop}
Using $\Psi$ as defined in \cref{eq:psi}, we have 
\begin{equation}
J \cdot (v \otimes w) = \Psi_{V \otimes W} (J_{V,W} (\Psi^{-1}_V(v) \otimes \Psi^{-1}_W(w)))
\end{equation}
for $v \in V\ph, w \in W\ph$.
\end{prop}
\begin{proof}
For each $v \in V\ph$, we define $f_v:M_+\otimes M_- \to V$ by $f_v (x) = \epsilon(v,x) x \cdot v$. Then $f_v(1_+ \otimes 1_-) = v$, since $1_+ \otimes 1_-$ is the unit in $M_+ \otimes M_-$.
So we have $f_v = \Psi^{-1}(v)$.
Let $\theta_1 \otimes \theta_2 := (\phi \otimes \phi )J \in (M_+\otimes M_-)^{\otimes 2}$.
Then the right hand side gives 
\begin{align}
 (J_{V,W} (\Psi^{-1}_V(v) \otimes \Psi^{-1}_W(w)))(1_+ \otimes 1_-) =  (f_v \otimes f_w) (\theta_1 \otimes \theta_2)
 = \epsilon(w,\theta_1) f_v (\theta_1) \otimes f_w(\theta_2) \\
 = \epsilon(w,\theta_1) \epsilon(v,\theta_1) \epsilon(w,\theta_2) \phi^{-1}(\theta_1)\cdot v \otimes \phi^{-1}(\theta_2)\cdot w.
\end{align}
And the left hand side gives 
\begin{align*}
 \phi^{-1}(\theta_1) \otimes \phi^{-1}(\theta_2) (v \otimes w)= \epsilon(\theta_2,v) \phi^{-1}(\theta_1)\cdot v \otimes \phi^{-1}(\theta_2)\cdot w.
\end{align*}
This equal since $J$ and with this $\theta_1 \otimes \theta_2$ are of degree 0.
\end{proof}

\begin{lemma}\label{th:Jlimit}
We have
\begin{equation}
J \equiv 1 + \frac{\lambda}{2}r \mod \lambda^2.
\end{equation}
\end{lemma}
\begin{proof}
Recall that $r = \sum m_i \otimes p_i$, where the $p_i$ are a basis of $g_+$ and the $m_i$ are the corresponding dual basis of $\g_-$. We have $\tau(r)\cdot (1_- \otimes 1_+) = 0$, since $p_i$ acts trivially on $1_+$. 
So we have using $\Phi \equiv 1 \mod \lambda^2$
\begin{align*}
J &\equiv ( \phi^{-1} \otimes \phi^{-1})(\E^{\lambda \Omega_{23}/2} (  1_+ \otimes 1_- \otimes 1_+ \otimes 1_-) \mod \lambda^2 \\
&\equiv 1 + \frac{\lambda}{2}( \phi^{-1} \otimes \phi^{-1})(r_{23} + \tau(r_{23}))(  1_+ \otimes 1_- \otimes 1_+ \otimes 1_-) \mod \lambda^2 \\
&\equiv 1 + \frac{\lambda}{2}( \phi^{-1} \otimes \phi^{-1})(  1_+ \otimes  p_i 1_- \otimes m_i 1_+ \otimes 1_-) \mod \lambda^2 \\
&\equiv 1 + \frac{\lambda}{2} p_i \phi^{-1} (  1_+ \otimes 1_-) \otimes m_i \phi^{-1}(1_+ \otimes 1_-) \mod \lambda^2 \\
&\equiv 1 + \frac{\lambda}{2} r \mod \lambda^2.
\end{align*}
\end{proof}

\begin{defn}
We can now define a color  Hopf algebra $H$ on $U(\g)\ph$ by 
\begin{align*}
\Delta= J^{-1} \Delta_0 J, \ \varepsilon = \varepsilon_0 \text{ and }
S= Q S_0 Q^{-1},
\end{align*}
with $ Q = \mu (S_0 \otimes \id)J$.
\end{defn}

we  give the corresponding color  Hopf algebra structure on $\End(V)$ under the isomorphisms $\Theta$. For this we first need:

\begin{lemma}
We have
\begin{equation}
\End(F) \otimes \End(F) = \End(F \otimes F)
\end{equation}
as algebra.
\end{lemma}
\begin{proof}
In the classical case this is obvious, but it is not so in our case, since we consider the enriched natural transformations. Here it follows from the definition as an end.
\end{proof}

There is a natural coproduct  on $\End(F)$ given by
\begin{equation}\label{eq:coprod}
\Delta(a)_{V,W} = J^{-1}_{V,W} a_{V \otimes W} J_{V,W} \text{ for } a \in \End(F)
\end{equation}
and $\End(F)$ becomes a bialgebra with it. It fact it is a Hopf algebra.
This follows directly from the fact that $(F,J)$ is a tensor functor, which gives that  $J$ is a twist. The twisted coproduct is precisely the one given here and the twisted $\Phi$ is trivial. We have that $H \cong \End(V)$ as Hopf algebra. 

\begin{prop}
The Hopf algebra $H$ is a  quantization of  the color Lie algebra $\g$. Moreover we define an $r$-matrix on $H$ by \begin{equation}
R= (J^\opp)^{-1} \E^{\frac{\lambda}{2}\Omega}.
\end{equation}
Then $(H,R)$ is a quasitriangular quantization of $(\g,r)$.
\end{prop}
\begin{proof}
 By definition $\factor{H}{\lambda H}$ is  isomorphic to $U(g)$ as color Hopf algebra.  From \Cref{th:Jlimit} and the definition of the coproduct it follows that $ 
\Delta_1(x) = \frac{1}{2} r \Delta_0(x) + \frac{1}{2}\Delta_0(x) r = \frac{1}{2}[\Delta_0(x),r]$, so we have
 $\delta(x) = \Delta_1(x) - \Delta_1^\opp(x) =  \frac{1}{2} [\Delta_0(x),r- \tau(r)] =  [\Delta_0(x),r]$, since $r+ \tau(r)= \Omega$, which is $\g$ invariant.

$R$ is an $r$-matrix, because it is obtained by twist, or explicit computation.
 
From \Cref{th:Jlimit}, we immediately get $R \equiv 1 \otimes 1 + \lambda( \frac{1}{2} r -  \frac{1}{2} \tau(r) +  \frac{1}{2} \Omega) \equiv  1 \otimes 1 + \lambda r \mod \lambda^2$.
\end{proof}

\subsection{Quantization of $\g_+$ and $\g_-$}

As shown before, we have $\End(F) \cong \End(M_+ \otimes M_-)$, since both are isomorphic to $U(g)$ as color algebras.  So we can define $U_\lambda(\g_+) =F(M_-)$ and embed it into $H$ via the map $i: F(M_-) \to \End(M_+ \otimes M_-)$  given by 
\begin{equation}
i(x) = (\id \otimes x) \circ \Phi \circ ( i_+ \otimes \id) 
\end{equation}
for $x  \in F(M_-)$. 
Then $i$ is injective, and satisfies 
\begin{equation}
i(x) \circ i(y) = i(z)
\end{equation}
for $x,y \in F(M_-)$ and  $z =  x \circ (\id \otimes y) \circ \Phi \circ (i_+ \otimes \id) \in F(M_-)$.
So $U_\lambda(\g_+)$ is a  color subalgebra of $H$.

We next want to show that it is indeed a color Hopf subalgebra, for this we need the following proposition.
\begin{prop}
The $r$-matrix of $H$ is polarized which means that $R \in U_\lambda(\g_+) \otimes U_\lambda(\g_-) \subset H \otimes H$.
\end{prop}
\begin{proof}
Following \citep[Lemma 19.4]{etingofbook}. 
The defining equation of $R$ is equivalent to
\begin{equation}\label{eq:defR2}
R \circ \beta^{-1}_{23} \circ ( i_+ \otimes i_-) = \beta_{23} \circ ( i_+ \otimes i_-)
\end{equation}
in $\Hom(M_+ \otimes M_-, M_+ \otimes M_- \otimes M_+ \otimes M_-)$, where we regard $R$ as an element in $\Hom(M_+ \otimes M_- \otimes M_+ \otimes M_-,M_+ \otimes M_- \otimes M_+ \otimes M_-)$.

Using the counit we have 
\begin{equation}
  \beta_{23} (i_+ \otimes i_-) = (\id \otimes \varepsilon \id \otimes \id \otimes \varepsilon \id \otimes) \beta_{34}(\id \otimes i_+ \otimes i_- \otimes  \id) (i_+ \otimes i_-) 
\end{equation}
and using \cref{eq:defR2} in the middle and then the coassociativity of $i_+$ and $i_-$ we further get
\begin{align*}
  \beta_{23} (i_+ \otimes i_-) &=  (\id \otimes \varepsilon \id \otimes \id \otimes \varepsilon \id \otimes)(\id \otimes R \otimes \id) \beta_{34} (\id \otimes i_+ \otimes i_- \otimes  \id) (i_+ \otimes i_-) \\
  &= (\id \otimes \varepsilon \id \otimes \id \otimes \varepsilon \id \otimes)(\id \otimes R \otimes \id) \beta_{34} (i_+ \otimes \id \otimes \id \otimes  i_-) (i_+ \otimes i_-) .
\end{align*}

So we have proved 
\begin{equation}
R = (\id \otimes \epsilon \otimes \id \otimes \id \otimes \epsilon \otimes \id) \circ  (\id \otimes R \otimes \id) \circ ( i_+ \otimes \id \otimes \id \otimes i_-).
\end{equation}

\end{proof}

We define maps $p_\pm: U_h(\g_\mp)^* \to U_h(\g_\pm)$ by 
\begin{equation}
p_+(f) = (\id \otimes f)(R) \text{ resp. } p_-(g) = (g \otimes \id)(R).
\end{equation} 
Let $\im p_\pm$ be the image of $p_\pm$ and $\U_\pm$ be the color  algebra generated by it, then we have 
\begin{lemma}
$\U_\pm$ is closed under the coproduct and the antipode.
\end{lemma}
\begin{proof}
Uses the hexagon identity. 
\end{proof}

\begin{lemma}
We have $U_h(\g_\pm) \otimes_{\K\ph} \K\laur\lambda = \U_\pm  \otimes_{\K\ph} \K\laur\lambda$.
\end{lemma}
\begin{proof}
  The proof in \citep[Lemma19.5]{etingofbook} works also in the color case.
\end{proof}

\begin{theorem}
$U_h(\g_\pm)$ is a color Hopf algebra and a quantization of $\g_+$.
\end{theorem}
\begin{proof}
From the previous lemmas it follows that, it is a color Hopf subalgebra of $U_h$. 
\end{proof}

\section{Quantization of triangular color Lie bialgebras}

Let $\mathfrak{a}$ be a not necessarily finite dimensional triangular color Lie bialgebra, then we define $\g_+ :=\{( 1 \otimes f)(r), f \in \mathfrak{a}^* \} $, 
$\g_- :=\{( f \otimes 1)(r), f \in \mathfrak{a}^* \} $ and $\g = \g_+ \oplus \g_-$.
One can identify $\g_-$  with $\g_+$ via the map $\chi(f)= ( f\otimes \id)(r)$.
Then one can define a Lie bracket on $\g$, such that for $x,y \in \g_\pm$ it is the Lie bracket in $\g_\pm$ and for $x \in \g_+, y \in \g\-$ it is defined by
\begin{equation}
[x,y] :=  (\ad^* x)(y) - \epsilon(x,y) (\ad^* y)(x)
.\end{equation}

One can define a map  $\pi: \g \to \mathfrak{a}$, such that the restriction to $\g_+$ and $\g_-$ is the embedding. With this one has 
\begin{equation}
\pi([x,y]) = [ \pi(x),\pi(y)].
\end{equation}

\begin{prop}
The bracket on $\g$ defined above actually is a color Lie bracket, and the natural pairing gives an invariant inner product, so in fact a color  Manin triple.
\end{prop}

Let $\M_\mathfrak{a}$ be the category of $\mathfrak{a}$-modules, with morphisms given by 
$\Hom(V,W) = \Hom_\mathfrak a(V,W)\ph$, this again can be viewed as a category  enriched over $\grVec$.  Using  morphisms $\pi$ one can  define the pullback functor $\pi^*: \M_\mathfrak{a} \to \M_g$ to the Drinfeld category of $\g$.  One can also pullback the monoidal structure along this functor.  

$\Omega := r + \tau( r)$ is $\g$ invariant, this is needed to pullback the  monoidal structure. 

Using the pullback functor  $\pi^*$  and the Verma modules defined before one can define a functor $F: \M_\mathfrak{a} \to \A$ by 
\begin{equation}
F(V) := \Hom_{\M_\h} (M_+ \otimes M_-, \pi^*(V)).
\end{equation}
$F$ is again isomorphic to the forgetful functor, and so we have $H :=\End(F) =  U(\mathfrak{a\ph})$.  In the same way as before one can define a tensor structure on $F$, and with this a deformed bialgebra on $H$. 

Note that if $\mathfrak{a}$ was in fact triangular then, we have $\Omega=0$ and the Hopf algebra is  also triangular.  

So essentially following the construction in the previous section, one gets the following
\begin{theorem}
Any quasitriangular color Lie bialgebra admits a quasitriangular quantization $U_h(\g)$, and if $\g$ is triangular, the quantization is also triangular.
\end{theorem}

\section{Second quantization of color Lie bialgebras}

\subsection{Topological spaces}

Let $F$ be a space of functions into a topological space, then the weak topology is the initial topology with respect to the evaluation maps. 
Let $V$ and $W$ be topological vector spaces.  We need a topology for $\Hom(V,W)$. We use the weak one, for which a basis is given by
\begin{equation}
\left\lbrace  f \in \Hom(V,W) | f(v_i) \in U_i, i=1, \ldots, n \right\rbrace_{U_1,\ldots,U_n,v_1,\ldots,v_n}, 
\end{equation}
where $U_i$ are open sets in $W$ and $v_i$ are elements in $V$.
Let $\K$ be a field of characteristic zero, with the discrete topology, and $V$ a topological vector space over $\K$ . Then its topology is called linear if the open subspaces form a basis of neighborhoods of 0. 

Let $V$ be a topological vector space with a linear topology, then $V$ is called separated if the map $V \to \projlim_{U \text{ open subspace}} ( \factor{V}{U})$ is injective, this is e.g. the case when $V$ is discrete, i.e. 0 is an open set.

All topological vector spaces we consider will be linear and separated, so we will just call them topological vector spaces.  

If $V$ is finite dimensional than the weak topology on $\Hom(V,\K)$ is the discrete topology. 

In general a neighborhood basis of zero is given by cofinite dimensional subspaces.  

A topological vector space is called complete if the map $$V \to \projlim_{U \text{ open subspace}} ( \factor{V}{U})$$  is surjective.

Let $V$ and $W$ be topological vectors spaces, we define the topological
tensor product by
\begin{equation}
V \hat \otimes W = \projlim \factor{V}{V'} \otimes \factor{W}{W'},
\end{equation}
where $V',W'$ run over open subspaces of $V$ (resp. $W$). With this we have:

\begin{prop}\label{th:montopo}
Complete vector spaces with continuous linear maps form a symmetric monoidal category.  
\end{prop}

\subsubsection*{Topology and Grading} 

A topological color vector space, is a color vector space $V = \bigoplus_{i \in \Gamma} V_i$, where each space $V_i$ is a topological vector space.

A linear map between topological color vector spaces is continuous if each homogeneous part is continuous.

For a graded vector space we say that it is complete, separated or has a given property if every space $V_i$ has this property.  The tensor product can also be defined by replacing the usual tensor product over vector spaces by the completed one. 
We note that the tensor product involves the direct sum $\bigoplus_{j \in \Gamma} V_i \otimes W_{i-j}$, for which a priori it is not clear whether it is complete. But it turns out that here since the considered  topologies are linear, this is the case. 

In fact using the construction, which defines graded vector spaces as functors, one just replaces the category of vector spaces by the category of complete vector spaces and gets a category of graded complete vector spaces, which is again monoidal due to \Cref{th:montopo}. 


%
%
%

\subsection{Manin triples}

Let $\la$  be a color  Lie bialgebra with discrete topology, i.e. each $\la_i$ is equipped with the discrete topology,  and $\la^*$ its dual, with the weak topology. Since $\la$ is discrete $\la^*$ is the full graded dual. The cocommutator defines a continuous Lie bracket on $\la^*$. We have a natural topology on $\la \oplus \la^*$ and the above defines a continuous Lie bracket with respect  to this topology.

Let $\g$ be a Lie algebra, with a nondegenerate inner product $\Sp{\cdot}{\cdot}$ and $\g_+$ and $\g_-$ be two isotropic Lie subalgebras, i.e. $\Sp{\g_+}{\g_+} = 0$, such that $\g = \g_+ \oplus \g_-$.  Then the inner product defines an embedding $\g_- \to \g^*_+$. 
To get a topology on $\g$, we equip $\g_+$ with the discrete topology and $\g_-$ with the weak topology.  If the Lie bracket on $\g$ is continuous in this topology we call $(\g,\g_+,\g_-)$ a Manin triple. 

To every color Lie bialgebra one can associate a color Manin triple by $(\la \oplus \la^* ,\la,\la^*)$, and conversely every color Manin triple gives a Lie bialgebra on $\g_+$.

\subsection{Equicontinuous $\g$-modules }

Let $M$ be a topological vector space and $\{A_x\}_{x\in X}$ be a family in $\End M$. Then $\{A_x\}_{x \in X}$ is equicontinuous if for all open neighborhoods $U$ of $0$ in $M$ there exists an open neighborhood $V$ such that $A_x V =U$ for all $x\in X$.  

\begin{defn}
Let $M$ be a complete topological color vector space. Then we call $M$ an equicontinuous $\g$-module if there is a continuous color Lie algebra morphism $\pi: \g \to \End(M)$ such that $\{ \pi(\g)\}$ is an equicontinuous family.
\end{defn}

For two equicontinuous $\g$-modules $M,N$, we have that $M \otimeshat N$ is again an equicontinuous $\g$-module. Further on $(V \otimes W) \otimes U$ can be identified with $V \otimes (W \otimes U)$ and $V \otimes W$ with $W \otimes V$ by the flip. So we can define the symmetric monoidal category $\M^e_0$ of equicontinuous $\g$-modules.

We define again the Verma modules $M_\pm$ by $M_\pm = \Ind^\g_{\g_\pm} 1 = U(\g) \otimes_{U(\g_\pm)} 1 $. 

\begin{lemma}
The module $M_-$ equipped with the discrete topology is an equicontinuous $\g$-module. 
\end{lemma}
\begin{proof}
This is true in the non-graded case, see e.g. \citep{ek1}, so it also holds in the color case since it can be checked in each degree.
\end{proof}

We want to define a topology on $M_+$, for this we first  define a topology on $U(\g_-)$.  We have  $U_n(\g_-) \cong \oplus_{k\leq n} S^k \g_-$.  We equip $S^k \g_-$ with the weak topology coming from the embedding into $(\g_+^{\otimes k})^*$.  This gives a topology on $U_n(\g_-)$. Finally we put on $U(\g_-)$, and with this on $M_+$, the topology coming from the inductive limit $\lim U_n(g)=U(g)$.

\begin{lemma}
For all $g\in \g$ the map $\pi_{M_+}(g):M_+ \to M_+$ is continuous.
\end{lemma}

Next we need a topology on $M_+^*$. For this we note that $U_n(\g_-)^* \cong \bigoplus_{k\leq n} S^k \g_+$, so we equip $U_n(\g_-)^*$ with the discrete topology. Since $U(g_-)^*$ is the projective limit of  $U_n(\g_-)^*$, it carries the corresponding topology. 

\begin{lemma}
$M_+^*$ is  an equicontinuous $\g$-module.
\end{lemma}

However $M_+$ is not equicontinuous in general.

There is  a Casimir element. 
It corresponds to the identity under the isomorphism of $\la^* \otimeshat \la \to \End(\la)$.

Let $\M^t_\g$ be the category of equicontinuous $\g$-modules  and morphisms 
\begin{equation}
\Hom_{\M^t_\g}(V,W) = \Hom_\g(V,W)\ph,
\end{equation}
where the $\Hom$ on the right denotes the continuous $\g$-module morphisms. 

We define a natural transformation $\gamma$ by $\gamma_{V,W} = \beta_{W,V}^{-1} \in \Hom(V \otimes W, W \otimes V)$ for $V,W \in \M^e$. 

Using the completed tensor product, we can define the structure of a braided monoidal category on $\M^t_\g$ using $\Phi$ and  $\gamma$ similarly to \Cref{ch:quant1}.

\subsection{Tensor functor $F$}

Let $V$ be a complete color space over $\K$. Then the space $V \otimeshat \K\ph$ is  again a complete color space, and carries a natural structure of a topological color $K\ph$-module. A $\K\ph$-module is called complete if it is isomorphic to $V\ph$ for a complete color space $V$.

Let $\A^c$ be the category of complete   color $\K\ph$-modules.  On complete $\K\ph$-modules we define a tensor product $V \tilde{\otimes} W =  \factor{V \otimeshat W}{<1 \otimes \lambda - \lambda\otimes  1>}$ . Then with this a color $\K\ph$-modules by $(V \tilde{\otimes} W)_i = \bigoplus_{j \in \Gamma} (V_j \tilde{\otimes}  W_{i-j})$.

We define a functor $F$ from the category $\M^t_\g$ of equicontinuous $\g$-modules to the category of complete color vector spaces $\A^c$, by
$V \mapsto \Hom(M_-,M_+^* \otimes V)$

We can define a comultiplication on $M_+^*$ by
\begin{equation}
i^*_+ : M^*_+ \otimes M^*_+ \to M^*_+ : i^*_+(f \otimes g)(x) := (f \otimes g)(i_+(x))
\end{equation}
for $f,g \in M^*_+$. This is continuous.

\begin{prop}
There is  isomorphisms $\Psi_V: F(V) \to V$ for all $V \in \M^t_\g$ natural in $V$, given by 
\begin{equation}
f \mapsto (\ev(1_+) \otimes \id) f (1_-)
\end{equation} 
\end{prop}
\begin{proof}
This follows from Frobenius reciprocity.
\end{proof}

This shows that $F$ is natural isomorphic to the forgetful functor.

We now want to define a tensor structure on the functor $F$. Similar to \eqref{eq:defJ} we define, $J_{V\otimes W} ( v \otimes w)$ by 
\begin{align*}
M_- \xrightarrow{i_-} M_- \otimes M_-  \xrightarrow{v \otimes w }
 (M_+^* \otimes V)  \otimes (M_+^* \otimes W) 
  \xrightarrow{\Phi}  M_+^* \otimes ((V  \otimes M_+^*) \otimes W)  \\
   \xrightarrow{\gamma_{23}}  M_+^* \otimes (( M_+^* \otimes V  ) \otimes W)
 \xrightarrow{\Phi} ( M_+^* \otimes  M_+^*) \otimes (V  \otimes W) 
  \xrightarrow{i_+^* \otimes \id \otimes \id}  M_+^*  \otimes (V  \otimes W) 
\end{align*}

That is (with out associators)
\begin{equation}
J_{VW} (v \otimes w)= (i_+^* \otimes \id \otimes \id ) \circ  (\id \otimes \gamma \otimes \id) \circ  (v \otimes w)  \circ i_-.
\end{equation}

Again one can write down a version without using explicitly  the maps $v, w$, which do not exist in the categorical sense for the graded case.

\begin{lemma}
We have
\begin{equation}
\Phi \circ (i_- \otimes \id) \circ  i_- = (\id \otimes i_-) \circ i_-
\end{equation}
and 
\begin{equation}
\Phi \circ (i_+^* \otimes \id) \circ  i_+^* = (\id \otimes i_+^*) \circ i_+^*
\end{equation}
i.e. $i_-$ and $i_+^*$ are coassociative in $\Hom(M_-,(M_-)^{\otimes 3})$ resp. $\Hom(M_+^*,(M_+^*)^{\otimes 3})$.
\end{lemma}

\begin{prop}
The maps $J_{VW}$ are isomorphisms and  define a tensor structure on $F$.
\end{prop}
\begin{proof}
They are isomorphisms because they are isomorphisms modulo $\lambda$. 
\end{proof}

\subsection{Quantization of color Lie bialgebras}

Let $H=\End(F)$ be the algebra of endomorphisms of the fiber  functor $F$, where $\End(F)$ is again to be understood in the enriched sense.
Let $H_0$ be the endomorphism algebra of the forgetful functor from $\M^e_0$ to the category of complete color vector spaces.  The algebra $H$ is naturally isomorphic to $H_0\ph$ .

Let $F^2: \M^e \times \M^e \to \A^c$ be the bifunctor defined by $F^2(V,W) =F(V \otimes W)$ and $H^2 = \End(F^2)$  then $H \otimes H \subset H^2$ but not necessarily $H^2 = H \otimes H$. 

$H$ has a ``comultiplication'' $\Delta: H \to H^2$ , defined by 
\begin{equation}
\Delta(a)_{VW}(v \otimes w)=  J_{VW}^{-1} a_{V\otimes W} J_{VW} (v \otimes w).
\end{equation}

\subsection{The algebra $U_h(g_+)$}

For $x \in F(M_-)$ we define $m_+(x) \in \End(F)$ by  
\begin{equation}
m_+(x)(v) = \epsilon(x,v) i_+^* \otimes \id \circ \Phi^{-1} \circ (\id \otimes v) \circ x
\end{equation}
for $v \in F(V)$.

We define $U_h(g_+) \subset H$ as the image of $m_+$. 
We have  $m_+$ is an embedding since it is so modulo  $\lambda$.

\begin{prop}
$U_h(\g_+)$ is a subalgebra of  $H$.
\end{prop}
\begin{proof}
Since $i_+^*$ is coassociative we get
\begin{align*}
&m_+(x)m_+(y)(v)
=  \epsilon(x,yv) \epsilon(y,v) (i^*_+ \otimes \id) \Phi^{-1} ( \id \otimes i^*_+ \otimes \id) \Phi^{-1}_{2,3,4}  (\id \otimes \id \otimes v)( \id \otimes y) x  \\
&=  \epsilon(x,y) \epsilon(x,v) \epsilon(y,v) (i^*_+ \otimes \id)  (  i^*_+ \otimes  \id \otimes\id)\Phi^{-1}_{1,2,3}\Phi^{-1}_{1,23,4} \Phi^{-1}_{2,3,4} (\id \otimes \id \otimes v)( \id \otimes y) x \\
&=  \epsilon(x,y)  \epsilon(xy,v) (i^*_+ \otimes \id)  (  i^*_+ \otimes \id \otimes\id)\Phi^{-1}_{12,3,4}\Phi^{-1}_{12,3,4}  (\id \otimes \id \otimes v)( \id \otimes y) x \\
&=  \epsilon(x,y)  \epsilon(xy,v) (i^*_+ \otimes \id) \Phi^{-1} (\id \otimes v) ( \otimes i^*_+ \otimes \id)\Phi^{-1}  ( \id \otimes y) x  \\
&=    \epsilon(xy,v) (i^*_+ \otimes \id) \Phi^{-1} (\id \otimes v) z \\
&= i(z) v,
\end{align*}
where $z = ( \otimes i^*_+ \otimes \id)\Phi^{-1}  ( \id \otimes y) x $.
\end{proof}

\begin{prop}
The algebra $U_h(\g_+)$ is closed under the coproduct. 
\end{prop}
\begin{proof}
The proof in \citep[Section 21.2]{etingofbook} is pictorial so it can also be used in the color case. 

The element $\Delta(m_+(x))$ is uniquely defined by the equation 
\begin{equation}
\begin{split}
(i_+^* \otimes \id \otimes \id) \circ (\id \otimes i_+^* \otimes \id \otimes \id) \circ \gamma_{34} \circ (\id \otimes v \otimes w) \circ (\id \otimes i_-) = \\
(i_+^* \otimes \id \otimes \id) \circ \gamma_{23}  \circ \Delta(m_+(x))(v \otimes w) \circ i_-
\end{split}
\end{equation}
for $v \in F(V), w \in F(W)$.

We want to get 
\begin{align*}
(i_+^* \otimes \id \otimes \id) \circ \gamma_{23} \circ (i_+^* \otimes \id \otimes i_+^* \otimes \id) \circ (\id \otimes v \otimes \id \otimes w) =\\
(i_+^* \otimes \id \otimes \id) \circ (i_+^* \otimes \id \otimes \id \otimes \id) \circ \gamma_{34} \circ (\id \otimes v \otimes w) \circ (i_+^* \otimes \id \otimes \id) \circ \gamma_{23} .
\end{align*}
\end{proof}

\begin{theorem}
$U_h(\g_+)$ is a quantization of $\g_+$, so for every Lie algebra there exists a quantized universal enveloping algebra.
\end{theorem}

\section{Simple color Lie bialgebras of Cartan type}

In the case of  Lie superalgebras, there are the so called classical simple ones of type A-G. 

Let $A = (A_{ij})_{{i,j}\in I}$ be a Cartan matrix, $I= \{1,\dots,s\}$ and $\tau \subset I$ the set corresponding to odd roots.

Let $\g$ be the Lie superalgebra generated by $h_i, e_i$ and $f_i$ for $i \in I$. We can put a 
$\Z^s$ grading on it as follows. We denote by $z_i$ the $i$-th generator.
The elements $h_i$ are all of degree zero and $\deg(f_i) = -\deg(e_i) = z_i$. So we consider $\g$ to be graded by the root system. 
The commutation factor is given by $\epsilon_0(z_i,z_j)=-1$ if either $i\in  \tau$ or $j \in \tau$ and $\epsilon(z_i,z_j)=1$ else. 

The generators satisfy the relations \citep{geer,yamane}
\begin{equation}
[h_i,h_j] = 0, [h_i,e_j] = A_{ij}e_j , [h_i,f_j] =-A_{ij} f_j , [e_i,e_j] = \delta_{ij} h_i
\end{equation}
and the so called super classical Serre-type relations 
\begin{align*}
[e_i,e_j] &= [f_i,f_j] = 0 \\
(\ad e_i)^{1+ \abs{A_{ij}}} e_j &= (\ad f_i)^{1+ \abs{A_{ij}}} f_j  = 0 \text{ if } i \neq j, i \notin \tau \\
[e_m,[e_{m-1},[e_m,e_{m+1}]]]& = [f_m,[f_{m-1},[f_m,f_{m+1}]]] = 0 \text{ for $m-1,m,m+1 \in I, A_{mm}=0$ }\\
[[[e_{m-1},e_m],e_m],e_m]& = [[[f_{m-1},f_m],f_m],f_m] =0 
	\\&\text{ if the Cartan matrix is of type B and $\tau = m,  s=m$.}
\end{align*}
The relations respect the $\Z^s$-grading so $\g$ can be considered as a $\Z^r$-graded algebra.

For a set of constants $\epsilon_{ij} \in \K^\times, i,j =1 ,\dots,s$. We define $\sigma(a ,b ) = \prod_{i,j=1}^s \epsilon_{ij} a_i b_j$ for $a,b \in \Z^r$. Then $\sigma: \Z^r \times \Z^r \to \K^\times$ is a bicharacter and $\epsilon'(a,b) = \sigma(a,b) \sigma^{1}(b,a)$ a commutation factor. 
We set $\epsilon = \epsilon' \epsilon_0$. 
For $x,y \in g$, we define $[x,y]' = \sigma(x,y) [x,y]$.  With this bracket $\g$ becomes a $\epsilon$ color Lie algebra. 
For more details on this construction see \citep{colorlie}.

There are $r$ matrices on these Lie algebras, see e.g. \citep{karaali}. Not all of them given there respect the $\Z^r$-grading, but the standard r-matrices given by $r = \sum h_i \otimes h_i + \sum_{\alpha \in \Delta^+} e_\alpha \otimes f_{\alpha}$ do. Here $\Delta^+$ denotes the set of positive roots. 

For these Lie superalgebras there is a well known quantization given by the so called Drinfeld-Jimbo type superalgebras. To define them we first need

\begin{equation}
\qbinom{m+n}{n} = \prod_{i=0}^{n-1} \frac{t^{m+n-i} -t^{-m-n+i}}{t^{i+1} - t^{-i-1}}.
\end{equation}
Assume that the Cartan matrix $A$ is symmetrizable that is there are non-zero rationals number $d_1,\dots, d_s$ such that $d_i A_{ij} = d_j A_{ji}$.
 Set $q = \E^{\lambda/2}$ and $q_i = \E^{d_i}$.
 
 Let $U(\g)$ be the $\C\ph$ superalgebra generated by $h_i, e_i$ and $f_i$, $i=1 ,\dots s$ and relations
\begin{align*}
 [h_i,h_j] &= 0, [h_i,e_j] = A_{ij}e_j , [h_i,f_j] =-A_{ij} f_j \\
 [e_i,e_j] &= \delta_{ij} \frac{q^{d_i h_i} - q^{-d_i h_i}}{q_i -q_i^{-1}}  \\
 e_i^2 &=0 \text{ for $i \in I, A_{ii}=0$} \\
 [e_i,e_j] &=0 , i,j \in I, i \neq j, A_{ij}=0 \\
 \sum_{k=0}^{1+\abs{A_{ij}}} (-1)^k &\qbinom{1+\abs{A_{ij}}}{k} e_i^{1+\abs{A_{ij}}-k} e_j e_i^k =0, 1\leq i, j\leq s, i \neq j, i \notin \tau \\
 e_{m} e_{m-1}e_{m} e_{m+1} &+  e_{m} e_{m+1}e_{m} e_{m-1} +  e_{m-1} e_{m}e_{m+1} e_{m} \\
 &+ e_{m+1} e_{m}e_{m-1} e_{m} (q + q^{-1}) e_{m} e_{m-1}e_{m+1} e_{m} =0 , m-1,m,m+1 \in I, A_{mm}=0 \\
 e_{m-1} e_{m-1}^3 &- ( q +q^{-1}-1) e_{m} e_{m-1}e_{m}^2 -  (q +q^{-1} -1)  e_{m}^2 e_{m-1}e_{m} 
+e_m^3 e_{m-1} = 0 \\ & \text{ if the Cartan matrix is of type B and $\tau = m,  s=m$},
 \end{align*}
and the same relations where $e_i$ is replaced by $f_i$.

For more details on this super quantized universal enveloping  algebra see e.g. \citep{yamane}.

Note that  all relations are compatible with the $\Z^r$-grading. 
Again we set $x y = \sigma(x,y) xy$, and get so a $\epsilon$-color algebra.

We  define a comultiplication on $U(\g)$  by specifying it on generators as 
\begin{align*}
\Delta(e_i) = e_i  \otimes q^{d_i h_i} + 1 \otimes e_i ,\\
\Delta(f_i) = f_i  \otimes q^{-d_i h_i} + 1 \otimes f_i ,\\
\Delta(h_i) = h_i \otimes 1 + 1 \otimes h_i ,\\
\epsilon(h_i) =\epsilon(f_i) = \epsilon(e_i) =0.
\end{align*}
Here the $\epsilon$ does not appear, but it does if one computes the comultiplication of other elements, since it appears in the definition of the multiplication on $U(\g) \otimes U(\g)$.

Therefore we get a color Hopf algebra structure on $U(\g)$. 

\def\cprime{$'$}

\end{document}